\documentclass{amsart}
\usepackage{amsmath,amsthm,amssymb}

\newcommand{\paren}[1]{\left(#1\right)}

\title[The automorphism group]{The automorphism group of a certain unbounded  non-hyperbolic domain}

\author{Hyeseon Kim, Ninh Van Thu and Atsushi Yamamori}

\address{Center for Geometry and its Applications, Pohang University of Science and Technology,
Pohang 790-784, Republic of Korea}
\email{hop222@snu.ac.kr, hyeseon@postech.ac.kr}

\address{Center for Geometry and its Applications, Pohang University of Science and Technology,
Pohang 790-784, Republic of Korea}
\email{thunv@postech.ac.kr}

\address{Center for Geometry and its Applications, Pohang University of Science and Technology,
Pohang 790-784, Republic of Korea}
\email{yamamori@postech.ac.kr}



\thanks{The research of the authors was supported by an NRF grant 2011-0030044(SRC-GAIA) of the Ministry of Education, The Republic of Korea}

\theoremstyle{definition}
\newtheorem{theorem}{Theorem}
\newtheorem{lemma}[theorem]{Lemma}

\newtheorem*{remark}{Remark}

\newtheorem*{acknowledgement}{Acknowledgement}
\usepackage[all]{xy}
\begin{document}
\begin{abstract}
In this paper we determine the automorphism group of the Fock-Bargmann-Hartogs domain $D_{n,m}$ in $\mathbb{C}^n\times\mathbb{C}^m$ which is defined by the inequality ${\|\zeta\|}^2<e^{-\mu{\|z\|}^2}$.
\end{abstract}

\maketitle

\section{Introduction}
For a complex manifold, the automorphism group is the set of all biholomorphic self maps that forms a group under the composition law. We denote by $\mathrm{Aut}(M)$ the automorphism group of a complex manifold $M$. For a domain in $\mathbb{C}^n~(n\geq2)$, the automorphism group is not easy to describe explicitly. For the bounded case, the automorphism groups of various domains are given in~\cite{ABP, Shimizu,Sunada}. The descriptions of the automorphism groups of hyperbolic domains are also given in ~\cite{Isaev, Isaev2, KTKim}. In particular, the methodology in~\cite{Isaev2} is using of the well-known minimal dimensions of irreducible, faithful representations of complex simple Lie algebras. The purpose of this article is to describe the $\mathrm{Aut}(D_{n,m})$ for the Fock-Bargmann-Hartogs domain $D_{n,m}$ which is defined by 
\begin{equation*}
D_{n,m}:=\{(z,\zeta)\in\mathbb{C}^n\times\mathbb{C}^m;\;{\|\zeta\|}^2<e^{-\mu{\|z\|}^2}\},\quad\mu>0.
\end{equation*}The main feature of this domain is that it is an unbounded strongly pseudoconvex domain which is not hyperbolic in the sense of Kobayashi. We note that, if $m=1$, then $D_{n,1}\backslash\{(z,0)\}$ especially reduces to an example in \cite{Isaev}. The arguments in this paper rely on the Cartan's theorem by using the Bergman representative mapping in~\cite{Ishi} and an explicit form of the Bergman kernel function for $D_{n,m}$ in~\cite{Y}. The celebrated Cartan's theorem is engaged with the bounded circular domains. Even if our domain $D_{n,m}$ is unbounded, the positivity of the Bergman kernel of $D_{n,m}$ and the positive definiteness of the Bergman metric at the origin, ensure the Cartan's theorem. In general, it is not easy to compute the Bergman kernel function for an unbounded domain without determining the complete orthonormal basis for the Bergman space. The polylogarithm function plays a crucial role in determining the Bergman kernel function of $D_{n,m}$~(see~Section~\ref{sec:FBH}).

In Section~\ref{sec:Cartan} we will describe the Cartan's theorem based on the Bergman representative mapping in spirit to \cite{Ishi}. Then we will investigate the Fock-Bargmann-Hartogs domain in Section~\ref{sec:FBH}; namely, polylogarithm function on $D_{n,m}$ and the Bergman kernel of $D_{n,m}$. In Section~\ref{sec:Auto}, we will determine the $\mathrm{Aut}(D_{n,m})$ as a main result.

\section{Cartan's Theorem Revisited}\label{sec:Cartan}
This section is devoted to the study of a celebrated  theorem of Cartan stating that any automorphism $f$ of bounded circular domain with $f(0)=0$ must be linear.
It is well-known that this theorem also holds if the domain is hyperbolic (cf.~\cite[Cor.~5.5.2]{Kob}).
However, in general, this theorem does not hold for unbounded case (cf.~\cite[2.1.4.~Examples]{R}).
The purpose of this section is to prove that two conditions on the Bergman kernel \eqref{B} and \eqref{T} imply Cartan's Theorem.
We will see in the next section that a certain unbounded non-hyperbolic domain also satisfies these conditions.

Let $D\subset \mathbb C^N$ be a circular domain (not necessarily bounded) with the Bergman kernel $K_D$. We suppose that the domain $D$ contains the origin.
Throughout this section we assume the following conditions.
\def\theenumi{\roman{enumi}}
\begin{enumerate}
\item $K_D(0,0)>0$,\label{B}
\item $T_D(0,0)$ is positive definite. \label{T}
\end{enumerate}
Here $T_D(z,w)$ is an $N\times N$ Hermitian matrix defined by
$$T_D (z,w):=
\begin{pmatrix} 
\dfrac{\partial^2 }{\partial \overline{w_1}\partial z_1}\log K_D(z,w) & \cdots & \dfrac{\partial^2 }{\partial \overline{w_1}\partial z_N}\log K_D(z,w)
\\ \vdots & \ddots & \vdots\\
\dfrac{\partial^2 }{\partial \overline{w_N}\partial z_1}\log K_D(z,w) & \cdots & \dfrac{\partial^2 }{\partial\overline{ w_N}\partial z_N}\log K_D(z,w)
\end{pmatrix}.$$
We note that the above conditions are always checked if $D$ is bounded.
Ishi and Kai \cite{Ishi} proved Cartan's theorem by using the notion of the Bergman representative mapping.
They considered bounded circular cases. However their proof is also applicable for an unbounded domain whenever its Bergman kernel has the properties \eqref{B} and \eqref{T}.
The proof proceeds along the same lines as that of \cite{Ishi}.
But for the convenience of the reader and in order to make clear the role of conditions \eqref{B} and \eqref{T}, we give an outline of the proof.

Let $\varphi: D\rightarrow D'$ be a biholomorphism onto a domain $D'$. 
We start with the following transformations laws under biholomorphisms:
\begin{align}
K_D(z,w)&=\overline{\det J(\varphi, w)} K_{D'}(\varphi(z) ,\varphi (w)) \det J(\varphi, z)  ,\label{eq:B}\\
T_D(z,w)&=\overline{ {}^t J(\varphi, w)} T_{D'}(\varphi(z) ,\varphi (w))  J(\varphi, z), \quad \mbox{if $K_D(z,w)\not=0$} \label{eq:T}.
\end{align}
Here $J(\varphi,z)$ is the Jacobian matrix of $\varphi = {}^t (\varphi_1, \ldots, \varphi_N)$ at $z$:
$$J_D (\varphi,z):=
\begin{pmatrix} 
\dfrac{\partial \varphi_1}{\partial z_1}(z)& \cdots & \dfrac{\partial \varphi_1}{\partial z_N}(z)
\\ \vdots & \ddots & \vdots\\
\dfrac{\partial \varphi_N}{\partial z_1}(z) & \cdots &\dfrac{\partial \varphi_N }{\partial z_N}(z) 
\end{pmatrix}.$$
The first formula is well-known. For the proof of the second formula, see \cite{Tuboi}.
Before introducing the Bergman representative mapping $\sigma_0^D$, we need the following lemma which ensures the well-definedness of $\sigma_0^D$.
\begin{lemma}[\mbox{\cite[Lemma 2.5]{Ishi}}] \label{lem1}
Let $D$ be a circular domain (not necessarily bounded) with conditions \eqref{B}  and \eqref{T}. Then we have the followings:
\quad\\
(a) The Bergman kernel $K_D(z,0)$ is a non-zero constant.\\
(b) The matrix $T_D(z,0)$ is a constant matrix, so that it equals the positive definite Hermitian matrix $T_D(0,0)$. 
\end{lemma}
For the proof of the assertion (a), we use the condition \eqref{B} and the formula \eqref{eq:B}.
For the proof of the assertion (b), we use the condition \eqref{T} and the formula \eqref{eq:T}.

This lemma gives us the well-definedness of the Bergman representative mapping $\sigma_0^D: D \rightarrow \mathbb C^N$:
\begin{align*}
\left. \sigma_0^D (z):= T_D(0,0)^{-1/2} \mbox{grad}_{\overline{w}} \log \dfrac{K_D (z,w)}{K_D(0,w)} \right|_{w=0}.
\end{align*}
Here we define
$$\mbox{grad}_{\overline{w}}  f(w) := {}^{t} \left( \dfrac{\partial f}{\partial \overline{w_1}} (w), \ldots,
\dfrac{\partial f}{\partial \overline{w_N}}(w)  \right),$$
for anti-holomorphic functions $f (w )$ on $D$. A crucial property of the Bergman representative mapping is that 
$\sigma_0^D$ is a linear mapping when $D$ is circular. Indeed we have
\begin{lemma}[\mbox{\cite[Proposition 2.6 (1)]{Ishi}}]\label{lem:linear}
Let $D$ be a circular domain (not necessarily bounded) with conditions \eqref{B} and \eqref{T}.
Then we have $\sigma_0^D (z) = T_D(0,0)^{1/2} z$. 
\end{lemma}

Let $\varphi: D\rightarrow D$ be an automorphism of a circular domain $D$ with $\varphi(0)=0$.
We define a unitary matrix 
$$L(\varphi,0)=T_D(0,0)^{-1/2}  \overline{{}^{t}  J(\varphi,0) }^{-1} T_D(0,0)^{1/2}.$$
For our purpose, we need one more lemma.
\begin{lemma}[\mbox{\cite[Proposition 2.1]{Ishi}}]\label{lem:comm}
Let $\varphi$  be an automorphism of a circular domain $D$ (not necessarily bounded) with conditions \eqref{B}, \eqref{T} and $\varphi(0)=0$.
Then one has $\sigma_0^D \circ \varphi = L(\varphi,0)\circ \sigma_0^D$.
In other words, the following diagram is commutative.
$$\xymatrix{
D \ar[d]_{\sigma_0^D} \ar[r]^\varphi_\sim \ar@{}[dr]|\circlearrowright & D \ar[d]^{\sigma_0^D} \\
\mathbb C^N \ar[r]^{L(\varphi,0)} & \mathbb C^N \\
}$$
\end{lemma}

Now we can prove Cartan's Theorem.
\begin{theorem}\label{th:CUT}
Let $D$ be a circular domain (not necessarily bounded) with conditions \eqref{B} and \eqref{T}.
If $\varphi \in \mbox{Aut}(D)$ and $\varphi (0)=0$, then $\varphi$ is linear.
\end{theorem}
\begin{proof}
By Lemma \ref{lem:linear} and Lemma \ref{lem:comm}, we have
$\varphi(z)= T_D(0,0)^{-1/2} L(\varphi,0) T_D(0,0)^{1/2} z$.
It is obviously linear.
\end{proof}

\section{Fock-Bargmann-Hartogs domain}\label{sec:FBH}
We give an example of unbounded non-hyperbolic circular domain which satisfies \eqref{B} and \eqref{T}.
In this section we consider the following Hartogs domain:
$$D_{n,m}:=\{ (z,\zeta)\in \mathbb C^n \times \mathbb C^m; \|\zeta\|^2 < e^{-\mu  \|z\|^2 } \}, \quad \mu>0,$$
which is called the Fock-Bargmann-Hartogs domain in \cite{Y}.
We note that $D_{n,m}$ contains $\{ (z,0) \in \mathbb C^n  \times \mathbb C^m \}  \cong \mathbb C^n$.
Thus there does not exist a bounded domain $D$ such that $D_{n,m}$ is biholomorphic to $D$ and it is not hyperbolic.
The aim of this section is to verify the conditions \eqref{B} and \eqref{T}  for $D_{n,m} $ by using an explicit form of $K_{D_{n,m}}$.
\subsection{Polylogarithm function}
An explicit form of the Bergman kernel of $D_{n,m}$ is expressed in terms of the polylogarithm function.
We introduce this function here.

Recall that the logarithm has the following series expansion:
$$-\log (1-t) =\sum_{k=1}^\infty \dfrac{t^k}{k}, \quad |t|<1.$$
The polylogarithm $Li_{s}(t)$ is defined as a natural generalization of the right hand side:
$$ Li_{s} (t)= \sum_{k=1}^\infty \dfrac{t^k}{k^s}, \quad |t|<1, s\in \mathbb C.$$
An important fact about $Li_s(t)$ is that it is a rational function of $t$ when the number $s$ is a negative integer.
Actually it is verified by $Li_0(t)=t/(1-t)$ and $\frac{d}{dt}Li_s(t)= Li_{s-1}(t)/t$.
The first few are given by
\begin{align*}
&Li_{-1}(t)=\dfrac{t}{(1-t) ^2};\quad Li_{-2}(t)=\dfrac{t^2 +t}{(1-t) ^3}; \quad 
Li_{-3}(t)=\dfrac{t^3 + 4t^2 +t}{(1-t) ^4};\\ 
&Li_{-4}(t)=\dfrac{t^4+ 11t^3 + 11t^2 +t}{(1-t) ^5};\quad Li_{-5}(t)= \dfrac{t^5 + 26 t^4 + 66 t^3 + 26 t^2 + t}{(1 - t)^6}.
\end{align*}
For any negative integer $s=-n$, the following closed form of $Li_{-n}(t)$ is known \cite[eq. 2.10c]{Cvi}:
\begin{align}\label{polylog}
Li_{-n} (t) = \sum_{j=0}^{n}\frac{ (-1)^{n+j}  j! S(1+n,1+j)   }{(1-t)^{j+1}}  ,
\end{align}
where $S(\cdot, \cdot)$ denotes the Stirling number of the second kind.
By using this expression \eqref{polylog}, we know that the $m$-th derivative of the polylogarithm has a form:
\begin{align}\label{Anm}
\dfrac{d^m}{dt^m} Li_{-n} (t) = \dfrac{A_{n,m}(t)  }{(1-t)^{n+m+1}  } ,
\end{align}
where $A_{n,m}(t)$ is given by
$$A_{n,m}(t) = m!\sum_{j=0}^{n} (-1)^{n+j}  (m+1)_j S(1+n,1+j)(1-t)^{n-j}.$$
Here $(x)_n$ is the Pochhammer symbol.
The positivity of the coefficients is known \cite[Lemma 4.3]{Y}:
\begin{lemma}\label{lem:positive}
All coefficients of $A_{n,m}(t) $ are positive.
\end{lemma}
\subsection{Bergman kernel}
The Bergman kernel of $D_{n,m}$ is computed in \cite{Y}:
\begin{align*}
K_{D_{n,m}}((z,\zeta),(z',\zeta'))=\dfrac{\mu^n e^{m\mu\langle z,z' \rangle} }{\pi^{n+m}} \dfrac{d^m}{dt^m} Li_{-n} (t)|_{t= e^{\mu\langle z,z' \rangle} \langle \zeta,\zeta' \rangle }.
\end{align*}
We remark that the Bergman kernel of $D_{1,1}$ was computed explicitly by G. Springer~\cite{Springer}.

Let us check the condition \eqref{B} and \eqref{T}.
For the condition \eqref{B}, it is enough to check $\frac{d^m}{dt^m} Li_{-n} (t)|_{t=0} >0$.
By \eqref{Anm}, it is equivalent to $A_{n,m}(0)>0$.
As explained in Lemma \ref{lem:positive}, all coefficients of $A_{n,m}(t) $ are positive.
Thus we know that $K_{D_{n,m}}(0,0)>0$.

Next we check the condition \eqref{T}. For the condition \eqref{T} we need the following lemma.
\begin{lemma}\label{lem:hess}
\noindent
\quad\\
$(1)$ The complex Hessian of $\log e^{m \mu \|z\|^2 } -(n+m+1) \log (1-  e^{\mu \|z\|^2}  \|\zeta \|^2 )$ at the origin is positive definite.\\
$(2)$ If $P$ is a real polynomial such that $P(0)\neq0$ and $P(0)P'(0)>0$, then the complex Hessian of $\log P(e^{\mu \|z\|^2}  \|\zeta \|^2)$ at the origin is positive semi-definite.
\end{lemma}
\begin{proof}
For every $1\leq i,k\leq n$, a direct computation shows that
\begin{equation*}
\begin{aligned}
&\left.\frac{\partial^2}{\partial z_i\partial\bar z_k}(\log(1-e^{\mu{\|z\|}^2}\|\zeta\|^2))\right|_{(z,\zeta)=0}=
\left.\frac{\partial^2}{\partial z_i\partial\bar\zeta_k}(\log(1-e^{\mu{\|z\|}^2}\|\zeta\|^2))\right|_{(z,\zeta)=0}=0;\\
&\left.\frac{\partial^2}{\partial\zeta_i\partial\bar\zeta_k}(\log(1-e^{\mu{\|z\|}^2}\|\zeta\|^2))\right|_{(z,\zeta)=0}=-\delta^k_i,
\end{aligned}
\end{equation*}where $\delta^k_i$ is the Kronecker delta. By observing the complex Hessian of the first term, we obtain that
\begin{equation*}
\left.\frac{\partial^2}{\partial z_i\partial\bar z_k}(\log(e^{m\mu{\|z\|}^2}))\right|_{(z,\zeta)=0}=m\mu\delta^k_i.
\end{equation*}Moreover, the other elements vanish at the origin. Thus, we complete the proof of the first assertion. For the second assertion, we define the polynomial $P$ by setting 
\begin{equation*}
P(e^{\mu{\|z\|}^2}{\|\zeta\|}^2)=\sum_{s=0}^{\mathrm{deg}P}c_s{(e^{\mu\|z\|^2}{\|\zeta\|}^2)}^s.
\end{equation*}Together with our assumption on the positiveness of $c_0$ and $c_1$, we deduce that
\begin{equation*}
\left.\frac{\partial^2}{\partial\zeta_i\partial\bar\zeta_k}\log(P(e^{\mu\|z\|^2}\|\zeta\|^2))\right|_{(z,\zeta)=0}=\frac{c_1}{c_0}\delta^k_i\geq0;
\end{equation*}the other elements of the complex Hessian vanish at the origin. Hence, the proof of the second assertion is complete.
\end{proof}
Combining Lemma \ref{lem:positive} and Lemma \ref{lem:hess}, we know that $T_{D_{n,m}} (0,0)$ is positive definite.
Thus we have checked the conditions \eqref{B} and \eqref{T} for our domain $D_{n,m}$.
The above argument, together with Theorem \ref{th:CUT}, implies Cartan's theorem for $D_{n,m}$:
\begin{theorem}
If $\varphi \in \mbox{Aut}(D_{n,m}) $ with $\varphi (0)=0$, then $\varphi$ is linear.
\end{theorem}

\section{Automorphism Group}\label{sec:Auto}
As an application of Section \ref{sec:Cartan} and Section \ref{sec:FBH}, we determine the automorphism group of $D_{n,m}$.
Put $\mathcal U=\{ (z,0) \in \mathbb C^n  \times \mathbb C^m \} \subset D_{n,m}$.
We start with the following observation:
\begin{lemma}\label{lem:inv}
For any $\varphi \in \mbox{Aut} (D_{n,m}) $, the space $\mathcal U$ is invariant under the mapping $\varphi$ (i.e. $\varphi(\mathcal U)\subset \mathcal U$) and $\varphi|_{\mathcal U}\in\mathrm{Aut}(\mathcal U)$.
\end{lemma}
\begin{proof}
Put $\varphi(z,0)=(\varphi_1(z), \varphi_2(z))$ and $\varphi_2(z)=(\varphi_{21}(z),\ldots, \varphi_{2m}(z)  )$. 
Then we have
\begin{align*}
\sum_{i=1}^m  |\varphi_{2i}(z)| ^2= \| \varphi_2 (z)  \|^2 < e^{-\mu \|\varphi_1 (z) \|^2 }\leq 1. 
\end{align*}
It follows that $\varphi_{2i}$ is bounded and holomorphic in $\mathbb C^n$ for all $1\leq i \leq m$.
Then Liouville's Theorem implies that $\varphi_{2i}$ is constant.
Since $\varphi_1$ is a non-constant entire function, $\varphi_1$ is unbounded. Therefore, there exists a sequence $\{z_k \}_{k\in \mathbb N}$ such that $\varphi_1(z_k)=\infty $ as $k \rightarrow \infty$.
Therefore $\varphi_2$ must be identically equal to zero. This proves $\varphi(\mathcal U)\subset \mathcal U$ and $\varphi|_{\mathcal U}\in\mathrm{Aut}(\mathcal U)$.
\end{proof}
Next lemma is the key point of our argument.
\begin{lemma}\label{lem:linear2}
If $\varphi \in \mbox{Aut} (D_{n,m})$ is linear, then $\varphi (z,\zeta) = (U z, U' \zeta )$ for some $U \in U(n), U' \in  U(m) $.
\end{lemma}
\begin{proof}
Since $\varphi$ is linear, the map $\varphi$ can be written as a matrix form; namely,
\begin{equation*}
\varphi(z,\zeta)=
\begin{pmatrix}
A&C\\
D&B\\
\end{pmatrix}
\begin{pmatrix}
z\\
\zeta\\
\end{pmatrix},
\end{equation*}where $A\in M_{n\times n}(\mathbb{C})$, $B\in M_{m\times m}(\mathbb{C})$, $C\in M_{n\times m}(\mathbb{C})$, and $D\in M_{m\times n}(\mathbb{C})$. We note that $\varphi|_{\mathcal U}\in\mathrm{Aut}(\mathcal U)$. Applying the preceding lemma, we have
\begin{equation*}
D=O;\quad\mathrm{det}A\neq0;\quad\mathrm{det}B\neq0,
\end{equation*}where $O$ is the $(n\times m)$ zero matrix. Then it follows that $\varphi(z,\zeta)=(Az+C\zeta,B\zeta)$ for all $(z,\zeta)\in D_{n,m}$.

Now we shall show that $A\in U(n)$. Let us choose an unit eigenvector $\tilde e$ of $B$ and the associated eigenvalue $\lambda$ of $B$ such that $B\tilde e=\lambda\tilde e$. Moreover, the linearity of $\varphi$ ensures that $\varphi(\partial D_{n.m})=\partial D_{n,m}$ as a set, it follows that
\begin{equation}\label{equation:unitary_1}
{|\lambda|}^2t^2=e^{-\mu{\|Az+tC(\tilde e)\|}^2}
\end{equation}for all $(z,t\tilde e)\in\partial D_{n,m}$. By inserting $z=0$ into \eqref{equation:unitary_1}, we obtain that
\begin{equation*}
{|\lambda|}^2=e^{-\mu{\|C(\tilde e)\|}^2}; 
\end{equation*}hence, \eqref{equation:unitary_1} can be rewritten as 
\begin{equation}\label{equation:unitary_2}
e^{-\mu\paren{{\|z\|}^2+{\|C(\tilde e)\|}^2}}=
e^{-\mu\paren{{\|Az\|}^2+e^{-\mu{\|z\|}^2}{\|C(\tilde e)\|}^2+2\langle Az,e^{-\frac{\mu}{2}{\|z\|}^2}C(\tilde e)\rangle}}
\end{equation}for all $z\in\mathbb{C}^n$. Since $\mu>0$, \eqref{equation:unitary_2} implies that
\begin{equation}\label{big_O_condition}
({\|Az\|}^2-{\|z\|}^2)+(e^{-\mu{\|z\|}^2}-1){\|C(\tilde e)\|}^2+2(1+\mathbf{O}({\|z\|}^2))\langle Az,C(\tilde e)\rangle=0
\end{equation}or
\begin{equation*}
\langle Az,C(\tilde e)\rangle=\mathbf{O}({\|z\|}^2)
\end{equation*}for all $z\in\mathbb{C}^n$, which ensures that $\langle Az,C(\tilde e)\rangle\equiv0$ on $\mathbb{C}^n$. Then  we obtain $C(\tilde e)=0$ which concludes that 
\begin{equation*}
{\|z\|}^2={\|Az\|}^2
\end{equation*}for all $z\in\mathbb{C}^n$.

We next show that $B\in U(m)$. To verify this assertion, we now fix an arbitrary $\zeta\in\mathbb{C}^m$ with $0<\|\zeta\|<1$. Let $(z,\zeta)\in\partial D_{n,m}$ and $(z_t,t\zeta)\in\partial D_{n,m}$ for any  $0<t<\frac{1}{\|\zeta\|}$. Then these settings imply that 
\begin{equation}\label{unitary_3}
e^{-\mu{\|z_t\|}^2}={\|t\zeta\|}^2=t^2{\|\zeta\|}^2=t^2e^{-\mu{\|z\|}^2}.
\end{equation}Moreover, if we let $\tilde z=Az$ for $(z,\zeta)\in\partial D_{n,m}$, then ${\|z\|}^2={\|A^{-1}\tilde z\|}^2={\|\tilde z\|}^2$. This shows that $(\tilde z,\zeta)\in\partial D_{n,m}$. Then it follows that
\begin{equation}\label{unitary_4}
\begin{split}
{\|B\zeta\|}^2
&=e^{-\mu{\|\tilde z+C\zeta\|}^2}\\
&=e^{-\mu\paren{{\|\tilde z\|}^2+{\|C\zeta\|}^2+2\langle\tilde z,C\zeta\rangle}}\\
&={\|\zeta\|}^2e^{-\mu\paren{{\|C\zeta\|}^2+2\langle\tilde z,C\zeta\rangle}}
\end{split}
\end{equation}for all $(\tilde z,\zeta)\in\mathbb{C}^n\times\mathbb{C}^m$ with ${\|\zeta\|}^2=e^{-\mu{\|\tilde z\|}^2}$. On combining \eqref{unitary_3} with \eqref{unitary_4}, we deduce that
\begin{equation}\label{unitary_5}
\begin{split}
\frac{{\|B\zeta\|}^2}{{\|\zeta\|}^2}=\frac{{\|B(t\zeta)\|}^2}{{\|t\zeta\|}^2}
&=e^{-\mu({\|C(t\zeta)\|}^2+2\langle\tilde z_t,C(t\zeta)\rangle)}\\
&=e^{-\mu(t^2{\|C\zeta\|}^2+2t\langle\tilde z_t,C\zeta\rangle)}
\end{split}
\end{equation}for all $0<t<\frac{1}{\|\zeta\|}$. Moreover, \eqref{unitary_3} ensures that
\begin{equation*}
{\|\tilde z_t\|}^2=\frac{1}{\mu}\log(\frac{1}{t^2})+{\|\tilde z\|}^2
=\frac{1}{\mu}\log(\frac{1}{t^2})+{\|z\|}^2;
\end{equation*}hence by the Cauchy-Schwarz's inequality we have
\begin{equation}\label{unitary_6}
\begin{split}
{|2t\langle\tilde z_t,C\zeta\rangle|}^2
&\leq4t^2{\|\tilde z_t\|}^2{\|C\zeta\|}^2\\
&=4t^2(\frac{1}{\mu}\log(\frac{1}{t^2})+{\|z\|}^2){\|C\zeta\|}^2\rightarrow0
\end{split}
\end{equation}as $t\rightarrow0^+$. Since $t^2{\|C\zeta\|}^2$ obviously tends to $0$ as $t\rightarrow0^+$, \eqref{unitary_5} and \eqref{unitary_6} together imply that $\|B\zeta\|=\|\zeta\|$ for all $\zeta\in\mathbb{C}^m$ with $0<\|\zeta\|<1$. Because $B$ is linear, it holds for all $\zeta\in\mathbb{C}^m$, and hence $B\in U(m)$. 

To complete our proof, it suffices to show that $C=O$. Since $B$ is unitary, we consider the following setting: Let $\{\lambda_j\}$ be the set of eigenvalues of $B$ and $\{e_j\}$ the pertaining set of orthonormal eigenvectors of $B$ such that $Be_j=\lambda_j e_j$ for each $j$. Furthermore, the set $\{e_j\}$ can be a basis of $\mathbb{C}^m$. The remaining proof is similar in spirit to \eqref{equation:unitary_2} and \eqref{big_O_condition}. Similar arguments to each $e_j$ show that
\begin{equation*}
\langle Az,C(e_j)\rangle=\mathbf{O}({\|z\|}^2)
\end{equation*}for all $z\in\mathbb{C}^n$. Thus we obtain $C(e_j)=0$ for each $e_j$, which completes the proof.
\end{proof}
Now we are ready to prove our main theorem.
\begin{theorem}
The automorphism group $\mbox{Aut} (D_{n,m})$ for $D_{n,m}$ is generated by the following maps:
\begin{align*}
&\varphi_U: (z,\zeta) \mapsto (Uz, \zeta) , \quad U \in U(n);\\
&\varphi_{U'}: (z,\zeta) \mapsto (z, U'\zeta) , \quad U' \in U(m);\\
&\varphi_{v}: (z,\zeta) \mapsto (z+v, e^{-\mu v^* z- \frac{\mu}{2} \|v\|^2 } \zeta)  ,\quad v\in \mathbb C^n.
\end{align*}
\end{theorem}
\begin{proof}
It is easy to see that the above mappings belong to $\mathrm{Aut}(D_{n,m})$. Let $\varphi$ be an automorphism of $D_{n,m}$ and put $\varphi(0,0)=(v_0,v_0')$.
Then Lemma \ref{lem:inv} implies $v_0'=0$.
Thus we obtain that $\varphi_{-v_0} \circ \varphi$ preserves the origin. Namely, we have $\varphi_{-v_0} \circ \varphi(0)=0$.
This, together with Theorem \ref{th:CUT} and Lemma \ref{lem:linear2}, implies that $\varphi_{-v_0} \circ \varphi = \varphi_{U'} \circ \varphi_U$.
Hence we conclude that  $\varphi = \varphi_{v_0} \circ \varphi_{U'} \circ \varphi_U$. This proves the theorem.
\end{proof}

\begin{remark}
The preceding theorem tells us that $\mathrm{dim}_{\mathbb{R}}\mathrm{Aut}(D_{n,m})=n^2+m^2+2n$. If we let $m=1$, then $\mathrm{dim}_{\mathbb{R}}\mathrm{Aut}(D_{n,1})=\paren{n+1}^2$. Moreover, $D_{n,1}\backslash\{(z,0)\in\mathbb{C}^n\times\mathbb{C}\}$ can be one of the models in the classification of connected hyperbolic manifolds of dimension $k\geq2$ with $\mathrm{dim}_{\mathbb{C}}\mathrm{Aut}=k^2$; precisely, $D_{n,1}\backslash\{(z,0)\in\mathbb{C}^n\times\mathbb{C}\}$ belongs to $D_{r,\theta}:=\{(\tilde z,z_{k})\in\mathbb{C}^{k-1}\times\mathbb{C};\;r e^{\theta{\|\tilde z\|}^2}<|z_{k}|<e^{\theta{\|\tilde z\|}^2}\}$, with either $\theta=1,\; 0<r<1$, or $\theta=-1,\;r=0$~(cf.~\cite{Isaev}).
\end{remark}

\acknowledgement{The authors would like to express their gratitude to Professor Hideyuki Ishi and Professor Kang-Tae Kim for helpful discussion. Especially, the authors thank the anonymous referee for valuable comments on this paper.}

\end{document}